\documentclass[11pt]{amsart}

\usepackage{amsfonts,epsfig}
\usepackage{latexsym}
\usepackage{amssymb}
\usepackage{amsmath}
\usepackage{amsthm}
\usepackage{graphics}
\usepackage[all]{xy}
\usepackage[T2A]{fontenc}
\usepackage{multirow}
\usepackage{hhline}
\usepackage{array, booktabs}
\usepackage{hyperref}
\usepackage{enumerate}

\usepackage{tabularray}
\usepackage{arydshln}

\usepackage{tikz}  
\usetikzlibrary{positioning}
\usepackage{ctable} 

\usepackage{tabularray}

\usepackage{longtable}
\usepackage{booktabs}
\usepackage{diagbox}

\newtheorem{defn}{Definition}

\newtheorem{lemma}[defn]{Lemma}

\newtheorem{thm}[defn]{Theorem}

\newtheorem{prop}[defn]{Proposition}

\theoremstyle{definition}
\newtheorem*{ack}{Acknowledgements}
\newtheorem{remark}[defn]{Remark}

\newcommand{\Q}{\mathbb Q}

\newcommand{\Z}{\mathbb Z}

\title[Squares in arithmetic progression]{Squares in arithmetic progression over\\ quadratic extensions of number fields}

\author{Enrique Gonz\'alez--Jim\'enez}
\address{Universidad Aut{\'o}noma de Madrid, Departamento de Matem{\'a}ticas, Madrid, Spain}
\email{enrique.gonzalez.jimenez@uam.es}

\thanks{The author is supported by Grant PID2022-138916NB-I00 funded by MCIN/AEI/ 10.13039/501100011033 and by ERDF A way of making Europe.}


\begin{document}
\date{\today}


\begin{abstract}
We study arithmetic progressions of squares over quadratic extensions of number fields. Using a method inspired by an approach of Mordell, we characterize such progressions as quadratic points on a genus $5$ curve. Specifically, we determine the set of $K$-quadratic points on this curve under certain conditions on the base field $K$. Our main results rely on the algebraic properties of specific elliptic curves after performing a base change to suitable number fields. As a consequence, we establish that, under appropriate assumptions, any non-elementary arithmetic progression of five or six squares properly defined over a quadratic extension of $K$ must be of a specific form. Moreover, we prove the non-existence of such progressions of length greater than six under these assumptions.
\end{abstract}
\maketitle
\section{Introduction}
The study of squares in arithmetic progression dates back to Fermat. Specifically, in 1640, Fermat claimed that there are no four squares in arithmetic progression over $\Q$. Euler did not prove this conjecture until 1780. 

The next natural question in this line of research addresses the behavior of arithmetic progressions of squares in number fields. {Recall that an arithmetic progression of squares over a number field $K$ consists of elements $a_1,\ldots,a_n \in K$ such that $a_i^2 - a_{i-1}^2 = a_{i+1}^2 - a_i^2$ for $i \in\{ 2,\ldots,n-1\}$.} Xarles \cite{X}  established a significant result in the study of arithmetic progressions of squares in number fields. He proved the existence of a function $S(d)$ for any positive integer $d$, where $S(d)$ depends only on $d$, such that no arithmetic progression of squares of length $S(d)$ exists in any number field of degree $d$. Notably, Xarles demonstrated that $S(2) = 6$ for quadratic fields. This work was later extended by Bremner and Siksek \cite{BS}, who proved that S(3) = 5 for cubic fields. In the specific case \( d = 2 \), the author, in collaboration with Steuding \cite{Gon1}, studied the conditions under which arithmetic progressions of four squares exist within a given quadratic field. Furthermore, in joint work with Xarles \cite{GJX}, the author established several criteria for determining the quadratic number fields that admit arithmetic progressions of five squares. 

Recently, the author, jointly with Tho \cite{GJTho}, studied this problem in the case \( d = 4 \), that is, over quartic number fields. In particular, if \( D \) is a square-free integer, we characterize, under certain conditions on \( D \), arithmetic progressions of squares over quadratic extensions of \( K = \mathbb{Q}(\sqrt{D}) \). This study was based on a method of Mordell \cite{Mordell}. 

Expanding upon the construction developed in that article, we extend this characterization to quadratic extensions of number fields. We note that Tho \cite{tho} has announced that he has addressed the case \( K = \mathbb{Q}(\zeta_8) \), where we denote by \(\zeta_m\) a primitive \(m\)-th root of unity

The main results of this article are based on the algebraic structure of the elliptic curves  
\begin{equation*}
\begin{array}{l}
E_0\,:\, y^2 = x^3 - x^2 - 9x + 9, \\
E_1\,:\, y^2 = x^3 - x^2 + x, \\
E_4\,:\, y^2 = x^3 - x^2 - 64x + 220, \\
E_6\,:\, y^2 = x^3 - x^2 - 4x + 4, \\
\end{array}
\end{equation*}
after performing a base change to specific number fields. For any elliptic curve $E$ and integer $D$, we denote by $E^D$ the $D$-quadratic twist of $E$. 

Throughout the rest of this article, given a number field $K$, we will establish results under the assumption that certain conditions hold. These conditions are as follows:\\
    \indent $\bullet$ $ \texttt{cond}_A(K)$: \( E^{\pm 1}_1(K) = E^{\pm 1}_1(\mathbb{Q}) \), \( E_4(K) = E_4(\mathbb{Q}) \), and \( E^{\pm 1}_6(K) = E^{\pm 1}_6(\mathbb{Q}) \).\\
    \indent $\bullet$ $ \texttt{cond}_B(K)$:  \( \operatorname{rank}_{\mathbb{Z}} E_1^{\pm 1}(K) = 0 \) and either \( [K:\mathbb{Q}] < 8 \) or \( [K:\mathbb{Q}] \) is odd.

We have restricted ourselves to the case \( [K:\mathbb{Q}] < 8 \) because computing the rank of elliptic curves over number fields of degree greater than $7$ becomes intractable.

Let $n$ be a positive integer and $K$ be a field. Let $a_1, \dots, a_n \in K$ such that $(a_1^2, a_2^2, \dots, a_n^2)$ forms an arithmetic progression of length $n$. We say that this arithmetic progression is equivalent to $(s^2 a_1^2,s^2 a_2^2,\dots,s^2a_n^2)$ for any $s\in K^*$, as well as to its reverse $(a_n^2, \dots, a_2^2, a_1^2)$.  The arithmetic progression $(a_1^2,a_2^2,\dots,a_n^2)$ is properly defined over a number field $K$ if $a_1,\dots,a_n\in K$  and  $\{a_1,\dots,a_n\}\not\subset F$ for any proper subfield $F$ of $K$. Finally, an arithmetic progression is elementary if it is equivalent to a constant arithmetic progression or if an element in the arithmetic progression is $0$.  

The main results of this article are as follows.

\begin{thm}\label{main}
Let \( K \) be a number field with $K\ne\Q(\zeta_{12})$. Suppose that either $\texttt{cond}_A(K)$ or $\texttt{cond}_B(K)$ holds. Let $L$ be a quadratic extension of $K$.  Then  
\begin{enumerate}
\item[(i)] If the class number of $K$ is $1$, then any non-elementary arithmetic  progression of five squares properly defined over $L$ is, up to equivalence, of the form {$\left(a_1^2,a_2^2,a_3^2,\alpha \, a_4^2,a_5^2\right)$, where $a_1,a_2,a_3$, $a_4,a_5,\alpha \in K$} and $\alpha$ is non-square.
\item[(ii)] If $K\ne \Q$ and \( E_0(K) = E_0(\Q) \), then there does not exist any non-elementary arithmetic progression of five squares properly defined over \( L \).
\end{enumerate}
\end{thm}

The previous result concerns number fields different from $\Q(\zeta_{12})$. The next result is devoted to the specific case $ \Q(\zeta_{12})$. Note that in this case, condition $\texttt{cond}_B(\Q(\zeta_{12}))$ is satisfied.
 
\begin{thm}\label{Qz12}
The only non-elementary arithmetic progression of five squares properly defined over a quadratic extension of $\Q(\zeta_{12})$, up to equivalence, is $(-3, -1, 1, 3, 5)$, which is defined over $\Q(\zeta_{12}, \sqrt{5})$.
\end{thm}

Finally, under the assumptions of the previous results, we establish the case of non-elementary arithmetic progressions of length greater than $5$.

\begin{thm}\label{six}
Let \( K \) be a number field, and suppose that either \( \texttt{cond}_A(K) \) or \( \texttt{cond}_B(K) \) holds. Let \( L \) be a quadratic extension of \( K \). Assume that either \( E_0(K) = E_0(\mathbb{Q}) \), or, if \( E_0(K) \neq E_0(\mathbb{Q}) \), the class number of \( K \) is \( 1 \). 
{ Then, up to equivalence, the only non-elementary arithmetic progression of squares of length greater than $5$ is $(-5,-3,-1,1,3,5)$, where $K = \mathbb{Q}(\zeta_{12})$ and $L = K(\sqrt{5})$.
}
\end{thm}

This article is organized as follows. In Section~\ref{sec_setting}, given a number field $K$, we describe arithmetic progressions of five squares as $K$-rational points on a genus $5$ curve $\mathcal{C}$. In particular, we prove Theorem~\ref{quad_points}, which provides a characterization of quadratic points on $\mathcal{C}$ over a number field $K$ under certain hypotheses. The proof of this theorem relies on Proposition~\ref{prop}, whose proof is given in Section~\ref{sec_proof_prop}. In Section~\ref{sec_growth}, we prove Proposition~\ref{propGrowth}, where we show that if $E$ is an elliptic curve defined over $\mathbb{Q}$ without complex multiplication, {having at least one rational $2$-torsion point}, and if $2$ is the unique exceptional prime, then the torsion subgroup of $E$ does not grow over any number field of odd degree. This result is of independent interest beyond the study of arithmetic progressions of squares, but it will play a key role in the cases of number field of odd degree. Using the characterization provided in Theorem~\ref{quad_points}, we prove Theorem~\ref{main} in Section~\ref{sec_proof_main} and in Section \ref{sec_proof_Qz12} for the specific case where $K=\Q(\zeta_{12})$. Section~\ref{sec_proof_six} is devoted to the proof of Theorem~\ref{six}. Finally, in the Appendix, we collect the relevant information on certain genus one curves needed in the proofs throughout the article.

This work extensively utilizes the computer algebra system \verb|Magma| \cite{magma}. The code verifying the computational claims presented in this article is available at \cite{repository}.

\section{Quadratic points on a genus $5$ curve over number fields}\label{sec_setting}
{
Let $L$ be a number field, and let $a_1,a_2,a_3,a_4,a_5 \in L$, not all zero. Then $a_1^2$, $a_2^2$, $a_3^2$, $a_4^2$, and $a_5^2$ are in arithmetic progression if and only if $a_2^2 - a_1^2 = a_3^2 - a_2^2, a_3^2 - a_2^2 = a_4^2 - a_3^2$, and $a_4^2 - a_3^2 = a_5^2 - a_4^2$. In this case, $[a_1:a_2:a_3:a_4:a_5]\in \mathcal C(L) \subset \mathbb{P}^4(L)$, where $\mathcal C$ is the genus $5$ curve defined by the system of equations
\[
\mathcal C :
\begin{cases}
X_0^{2}+X_4^{2}=2 X_2^{2},\\[2mm]
X_0^{2}+X_2^{2}=2 X_1^{2},\\[2mm]
X_2^{2}+X_4^{2}=2 X_3^{2}.
\end{cases}
\]}
The first equation of $\mathcal C$  can be parametrized over $L$; that is, there exists \( t \in L \) such that
\begin{equation}\label{eqPara}
[X_0: X_2: X_4] = \left[\pm(t^{2}-2 t-1): \pm (t^{2}+1): \pm (t^{2}+2 t-1)\right].
\end{equation}
Thus, $\mathcal{C}$ can be described  as  
\begin{equation}\label{eq2C}
\mathcal{C} \,:\,  
\begin{cases}
X_1^{2} = G(t),\\[1mm]
X_3^{2} = G(-t),  
\end{cases}
\end{equation}
where $G(x) = x^{4} - 2x^{3} + 2x^{2} + 2x + 1$.  Note that the genus one curve \(C_0: y^2 = G(x)\) admits two morphisms \(\varphi_0, \varphi'_0: \mathcal C \longrightarrow C_0\) defined by \(\varphi_0(t, X_1, X_3) = (t, X_1)\) and \(\varphi'_0(t, X_1, X_3) = (-t, X_3)\). In other words, a point \( (t, X_1, X_3) \in \mathcal{C}(L) \) if and only if \( (t, X_1), (-t, X_3) \in C_0(L) \).

If $t\in\{0,\pm 1\}$, we obtain the points $[\pm 1:\pm 1: \pm 1 :\pm 1 : \pm 1] \in \mathcal{C}$.  We say that a point $[X_0:X_1:X_2:X_3:X_4] \in \mathcal{C}$ is \emph{trivial} if it corresponds to $t \in \{0, \pm 1\}$ or if $X_j = 0$ for some $j \in \{0, \dots, 4\}$. The following table lists these points along with their corresponding values of $t$:

\begin{longtblr}
[caption = {Trivial points in $\mathcal C$}, label=trivial_points]
{cells = {mode=imath},hlines,vlines,colspec  = ccc}
 t & [X_0:X_1:X_2:X_3:X_4]  & \\
0,\pm 1 & [\pm 1:\pm 1: \pm 1 :\pm 1 : \pm 1] & \\
1\pm \sqrt{2} & [0:\pm 1:\pm \sqrt{2}:\pm \sqrt{3}:\pm 2]  & X_0=0\\
\displaystyle \left(\frac{1\pm \sqrt{-1}}{2}\right)(1\pm\sqrt{3}) & [ \pm  \sqrt{-1}: 0: \pm 1: \pm \sqrt{2}: \pm \sqrt{3}]  & X_1=0\\
\pm \sqrt{-1}& [ \pm  \sqrt{-2}: \pm  \sqrt{-1}: 0: \pm 1: \pm \sqrt{2}] & X_2=0 \\
\displaystyle \left(\frac{-1\mp \sqrt{-1}}{2}\right)(1\pm\sqrt{3}) & [ \pm \sqrt{3}: \pm \sqrt{2} : \pm 1 : 0 :\pm  \sqrt{-1}]  & X_3=0\\
-1\pm \sqrt{2} & [\pm 2:\pm \sqrt{3} : \pm \sqrt{2} : \pm  1 : 0] & X_4=0\\
\end{longtblr}      

Other remarkable points in $\mathcal{C}$ are the following:
\newpage 
\begin{longtblr}
[caption = {Other points in $\mathcal C$}, label=otherpoints]
{cells = {mode=imath},hlines,vlines,colspec  = cc}
 t & [X_0:X_1:X_2:X_3:X_4]   \\
 \frac{1}{4}(1\pm \sqrt{5})(1+\sqrt{-3}) & [\pm \sqrt{3}:\pm 1: \pm \sqrt{-1} :\pm \sqrt{-3} : \pm \sqrt{-5}]  \\
 \frac{-1}{4}(1\pm \sqrt{5})(1+\sqrt{-3}) & [\pm \sqrt{-5} : \pm \sqrt{-3}: \pm \sqrt{-1}:\pm 1: \pm \sqrt{3}] \\
 \frac{1}{4}(1\pm \sqrt{5})(1-\sqrt{-3}) & [\pm \sqrt{-3}: \pm \sqrt{-1} :\pm 1 :\pm \sqrt{3} : \pm \sqrt{5}]  \\
 \frac{-1}{4}(1\pm \sqrt{5})(1-\sqrt{-3}) &  [ \pm \sqrt{5} :\pm \sqrt{3} :\pm 1: \pm \sqrt{-1}: \pm \sqrt{-3}]
\end{longtblr}

Now, assume $t\not\in\{0,\pm 1\}$ and define $s = s(t) = t - \frac{1}{t}$. From \eqref{eq2C}, we obtain  
\begin{equation}\label{eqGundert2}
\frac{G(\pm t)}{t^2} = t^{2}+\frac{1}{t^{2}}\mp 2\left(t-\frac{1}{t}\right)+2 = s^2\mp 2 s+4 \in L^{2}.
\end{equation}
On the other hand, from the definition of $s$, we also have  
$$
s^{2}+4=\left(t+\frac{1}{t}\right)^{2} \in L^{2}.
$$
This allows us to deduce that  
\[
\begin{cases}
(s^{2}+4)(s^{2}-2s+4)(s^{2}+2s+4)=\left(t+\frac{1}{t}\right)^{2} \frac{X_1^2}{t^2}\frac{X_3^2}{t^2} \in L^{2},\\[1mm]  
s^{2}(s^{2}-2s+4)(s^{2}+2s+4)=\left(t - \frac{1}{t}\right)^{2} \frac{X_1^2}{t^2}\frac{X_3^2}{t^2} \in L^{2}.
\end{cases}
\]
Defining $r = r(t) = s^{2}$, the previous equations become  
$$
\begin{cases}
(r+4)(r^{2}+4r+16) \in L^{2}, \\  
r(r^{2}+4r+16) \in L^{2}.
\end{cases}
$$

The previous constructions allow us to obtain six genus one curves \( C_i \), \( i=1,\dots,6 \), together with the corresponding morphisms \( \varphi_i: \mathcal{C} \longrightarrow C_i \), which are shown in Table~\ref{TableC}.

\begin{longtblr}
[caption = {Curves and morphisms}, label=TableC]
{cells = {mode=imath},hlines,vlines,colspec  = clc,rowhead=1}
C_i & \text{Equation} & \varphi_i(t,X_1,X_3) \\
C_1 & y^2 = (x+4)(x^2+4x+16)  & \displaystyle \left(r,\left(t^{-1}+t^{-3}\right)X_1X_3\right)  \\
C_2 & y^2 = x(x^2+4x+16)  &  \displaystyle \left(r,\left(t^{-1}-t^{-3}\right)X_1X_3\right)  \\
C_3 & y^2 = x(x+4)(x^2+4x+16) & \displaystyle \left(r,\left(1-t^{-4}\right)X_1X_3\right)  \\
 C_4 & y^2 = x(x^2+14x+1) &  (t^4,t^2X_1 X_3)\\
 C_5 & y^2 = x^4+4x^2+16 & (s,t^{-2}X_1 X_3)\\
 C_6  & y^2 = x^4+14x^2+1  & (t^2, X_1 X_3) \\
\end{longtblr}      

In this article we  determine the set of $K$-quadratic points on $\mathcal C$, given by 
$$
\Gamma_2(\mathcal C,K)=\bigcup\{\mathcal C(L) \,:\,	K\subset  L\subset \overline{\Q}\,\,\mbox{and}\,\,[L:K]=2\},
$$
under certain hypotheses on the field \( K \).

Let \( F \) be a number field of class number \( 1 \). If \( (t,u) \in C_0(F) \), we define
\[
\mathcal{P}_1(t,u) := \left\{ [\pm(t^{2}-2t-1) : \pm u : \pm(t^{2}+1) : \pm v \sqrt{m} : \pm(t^{2}+2t-1)] \right\}\subset \mathbb P^4(F),
\]
where $G(-t)=m\, v^2$; and if \( (-t,v) \in C_0(F) \), we define
\[
\mathcal{P}_3(-t,v) := \left\{ [\pm(t^{2}+2t-1) : \pm u \sqrt{n} : \pm(t^{2}+1) : \pm v : \pm(t^{2}-2t-1)] \right\}\subset \mathbb P^4(F),
\]
where $G(t)=n\, u^2$. We then define the sets
\[
\mathcal{S}_1(F) := \!\!\!\!\!\!\bigcup_{(t,u) \in C_0(F)}\!\!\!\!\!\! \mathcal{P}_1(t,u)
\quad \text{and} \quad
\mathcal{S}_3(F) := \!\!\!\!\!\!\bigcup_{(-t,v) \in C_0(F)}\!\!\!\!\!\! \mathcal{P}_3(-t,v).
\]
For clarity of exposition, we shall denote by \( \Gamma_2^*(\mathcal{C},F) \), \( \mathcal{S}_1^*(F) \), and \( \mathcal{S}_3^*(F) \) the sets \( \Gamma_2(\mathcal{C},F) \), \( \mathcal{S}_1(F) \), and \( \mathcal{S}_3(F) \), respectively, after removing the trivial points listed in Table~\ref{trivial_points}. We define the sets $\mathcal{T}$ and $\mathcal{P}_{\mathcal{T}}$, which arise from Table~\ref{otherpoints}. The set
\[
\mathcal{T} := \{\pm \frac{1}{4}(1 \pm \sqrt{5})(1 \pm \sqrt{-3})\}
\]
corresponds to the first column of Table~\ref{otherpoints}. On the other hand, the set
\[
\mathcal{P}_{\mathcal T}:=\left\{
\begin{array}{c}
[\pm \sqrt{3} : \pm 1 : \pm \sqrt{-1} : \pm \sqrt{-3} : \pm \sqrt{-5}],\quad [\pm \sqrt{-5} : \pm \sqrt{-3} : \pm \sqrt{-1} : \pm 1 : \pm \sqrt{3}],
\\[1mm]
[\pm \sqrt{-3} : \pm \sqrt{-1} : \pm 1 : \pm \sqrt{3} : \pm \sqrt{5}],\quad [\pm \sqrt{5} : \pm \sqrt{3} : \pm 1 : \pm \sqrt{-1} : \pm \sqrt{-3}]
\end{array}
\right\}.
\]
consists of the non-trivial points on the curve $\mathcal{C}$ over $\overline{\mathbb{Q}}$ that appear in the second column of Table~\ref{otherpoints} and are associated with the values in $\mathcal{T}$.

\begin{thm}\label{quad_points}
Let \( K \) be a number field with \( K \ne \mathbb{Q}(\zeta_{12}) \). Suppose that either \( \texttt{cond}_A(K) \) or \( \texttt{cond}_B(K) \) holds. Then:
\begin{itemize}
    \item[(i)] If the class number of \( K \) is \( 1 \), then $ \Gamma^*_2(\mathcal{C},K) = \mathcal{S}^*_1(K) \cup \mathcal{S}^*_3(K)$.
    \item[(ii)] If \( E_0(K) = E_0(\mathbb{Q}) \), then $ \Gamma^*_2(\mathcal{C},K) = \mathcal{S}^*_1(\mathbb{Q}) \cup \mathcal{S}^*_3(\mathbb{Q})$.
\end{itemize}
{
In the particular case \( K = \mathbb{Q}(\zeta_{12}) \), we have $\displaystyle 
\Gamma^*_2(\mathcal{C},\mathbb{Q}(\zeta_{12})) = \mathcal{S}^*_1(\mathbb{Q}) \cup \mathcal{S}^*_3(\mathbb{Q}) \cup \mathcal{P}_{\mathcal T}.$
}
\end{thm}

The proof of Theorem~\ref{quad_points} relies on the following result.

\begin{prop}\label{prop}
Let \( K \) be a number field, and suppose that either \( \texttt{cond}_A(K) \) or \( \texttt{cond}_B(K) \) holds. Let \( L \) be a quadratic extension of \( K \), {and  let \( (t, X_1, X_3) \in \mathcal{C}(L) \) be a non-trivial point such that \( t \notin\mathcal T \)}. Then, either \( (t, X_1) \in C_0(K) \) or \( (-t, X_3) \in C_0(K) \).
\end{prop}

\section{A Result on the Non-Growth of Torsion}\label{sec_growth}
In this section, we will prove a result concerning the torsion growth of elliptic curves under certain hypotheses. This result is of independent interest. However, it is the key piece for deriving the main results presented in the introduction in the case of number fields with odd degree.

{Let $E$ be an elliptic curve defined over $\mathbb{Q}$. Let $p$ be a prime, and let $G_E(p)$ denote the image (up to conjugacy in $\operatorname{GL}_2(\mathbb{F}_p)$) of the mod $p$ Galois representation on the $p$-torsion of $E$. We say that $p$ is \emph{exceptional} for $E$ if $G_E(p) \neq \operatorname{GL}_2(\mathbb{F}_p)$. In the case $p=2$, it is well known (cf. \cite{serre,sutherland,zywina}) that $G_E(2)$ is one of the following groups:
\[
 \operatorname{GL}_2(\mathbb{F}_2),\quad B(2), \quad \mathrm{C}_{ns}(2), \quad \text{or} \quad \mathrm{C}_s(2),
\]
where $B(2)$ (resp. $\mathrm{C}_{ns}(2)$, $\mathrm{C}_s(2)$) denotes the Borel (resp. non-split Cartan, split Cartan) subgroup of $\operatorname{GL}_2(\mathbb{F}_2)$. Note that $\mathrm{C}_s(2)$ is the trivial group.

\begin{lemma}\label{lemmaGrowth}
Let $E$ be an elliptic curve defined over $\mathbb{Q}$ without complex multiplication, and let $P \in E(\overline{\mathbb{Q}})$ be a point of order a power of $2$.  
\begin{itemize}
\item[(i)] If $G_E(2)$ is $B(2)$ or $\mathrm{C}_s(2)$, then $[\mathbb{Q}(P):\mathbb{Q}] = 2^n$ for some $n \geq 0$.
\item[(ii)] If $G_E(2)$ is $\operatorname{GL}_2(\mathbb{F}_2)$ or $\mathrm{C}_{ns}(2)$, then $[\mathbb{Q}(P):\mathbb{Q}] = 3\cdot 2^n$ for some $n \geq 0$.
\end{itemize}
\end{lemma}

\begin{proof}
Assume that $\mathrm{ord}(P) = 2^m$ for some $m \ge 1$.  If $m=1$ (see \cite[Table~1]{GJN1}):
\begin{equation}\label{eqp2}
[\mathbb{Q}(P):\mathbb{Q}] =
\begin{cases}
1 & \text{if } G_E(2) = \mathrm{C}_s(2),\\
1 \text{ or } 2 & \text{if } G_E(2) = B(2),\\
3 & \text{if } G_E(2) = \operatorname{GL}_2(\mathbb{F}_2) \text{ or } \mathrm{C}_{ns}(2).
\end{cases}
\end{equation}
If $m>1$, consider the tower of fields
\[
\mathbb{Q} = \mathbb{Q}(2^m P) \subseteq \mathbb{Q}(2^{m-1} P) \subseteq \cdots \subseteq \mathbb{Q}(2P) \subseteq \mathbb{Q}(P),
\]
so that
\[
[\mathbb{Q}(P):\mathbb{Q}] = \prod_{k=0}^{m-1} [\mathbb{Q}(2^k P): \mathbb{Q}(2^{k+1} P)].
\]
By~\cite[Theorem~4.6]{GJN1}, for $k \ne m-1$ we have $[\mathbb{Q}(2^k P): \mathbb{Q}(2^{k+1} P)] \in \{1,2,4\}$, and hence
\[
[\mathbb{Q}(P):\mathbb{Q}] = 2^s \cdot [\mathbb{Q}(2^{m-1} P):\mathbb{Q}]
\]
for some integer $s$. Finally, since $2^{m-1} P$ is a point of order $2$, the result follows from ~\eqref{eqp2}.
\end{proof}

\begin{prop}\label{propGrowth}
Let $E$ be an elliptic curve defined over $\mathbb{Q}$ without complex multiplication, with at least one rational $2$-torsion point, and such that $2$ is the unique exceptional prime. Then
\[
|E(\mathbb{Q})_{\mathrm{tors}}| = 2^n \quad \text{for some } n \in \{1,2,3,4\}.
\]
Moreover, if $K$ is a number field of odd degree, then $E(K)_{\mathrm{tors}} = E(\mathbb{Q})_{\mathrm{tors}}$.
\end{prop}

\begin{proof}
Since $2$ is the unique exceptional prime for $E$, and $E$ has at least one rational $2$-torsion point, it follows that
\begin{itemize}
    \item $G_E(2)$ is either $B(2)$ or $\mathrm{C}_s(2)$ (see~\eqref{eqp2}),
    \item $G_E(p) = \operatorname{GL}_2(\mathbb{F}_p)$ for every odd prime $p$.
\end{itemize}
In particular, if $P \in E(\overline{\mathbb{Q}})$ is a point of odd prime order $p$, then $[\mathbb{Q}(P) : \mathbb{Q}] = p^2 - 1$ by~\cite[Theorem~5.1]{Alvaro}. Since $p^2 - 1$ is even, it follows that such a point cannot be defined over any number field of odd degree. In particular, $E(\mathbb{Q})$ does not contain torsion points of odd order. By Mazur’s classification~\cite{mazur}, we therefore have $|E(\mathbb{Q})_{\mathrm{tors}}| = 2^n$ for some $n \in \{1,2,3,4\}$.

We now proceed to prove the second part of the statement. Let $K$ be a number field of odd degree. As shown above, $E(K)_{\mathrm{tors}}$ cannot contain points of order divisible by an odd prime. Moreover, by Lemma~\ref{lemmaGrowth} (i), no new torsion points of $2$-power order can appear over number fields of odd degree. Hence the torsion subgroup does not grow from $\mathbb{Q}$ to $K$, and we conclude that
\[
E(K)_{\mathrm{tors}} = E(\mathbb{Q})_{\mathrm{tors}}.
\]
This completes the proof.
\end{proof}
}

\begin{remark}\label{remark6}
For the elliptic curves \(E_0, E_1^{\pm 1}, E_4\) and \(E_6^{\pm 1}\), the only exceptional prime is \(2\). Therefore, Proposition~\ref{propGrowth} shows that if \(K\) is a number field of odd degree, then the torsion subgroup does not grow over \(K\) for any of these elliptic curves. In particular, if {\( \operatorname{rank}_{\Z} E_1^{\pm 1}(K) = 0 \)}, then \texttt{cond}$_A(K)$ holds, since \( E_1, E_4 \), and \( E_6 \) belong to the same \( \Q \)-isogeny class.
\end{remark}

\section{Proof of Proposition \ref{prop}}\label{sec_proof_prop}
Let \( K \) be a number field and let \( L \) be a quadratic extension of \( K \). Suppose that \( (t, X_1, X_3) \in \mathcal{C}(L) \) is a point that does not appear in Tables~\ref{trivial_points} or~\ref{otherpoints}. We aim to show that either \( (t, X_1) \in C_0(K) \) or \( (-t, X_3) \in C_0(K) \).

To this end, we divide the proof into several steps. First, we prove that \( r = r(t) \in K \); then we establish that \( s = s(t) \in K \); afterwards, we deduce that \( t \in K \); and finally, we conclude that either \( (t, X_1) \in C_0(K) \) or \( (-t, X_3) \in C_0(K) \).

\begin{remark}\label{rem2}
The proof relies on the computation of the set of $K$-rational points on the genus one curves $C_i$, for $i=0,1,\dots,6$. Table~\ref{TableEC} displays the corresponding $\Q$-isomorphic elliptic curve, namely $E_0, E_1, E_1^{-1}, E_1, E_4, E_6^{-1}$ and $E_6$, respectively.

Observe that, since $\operatorname{rank}_{\mathbb{Z}} E_1^{\pm 1}(\Q) = 0$, and the curves $E_1, E_4, E_6$ are $\Q$-isogenous, the conditions $\texttt{cond}_A(K)$ and $\texttt{cond}_B(K)$ reduce to controlling the torsion subgroups of these curves, and their $(-1)$-quadratic twists, over~$K$. In particular, we have the following:\\
\indent $\bullet$ If $\texttt{cond}_A(K)$ holds, then $C_i(K) = C_i(\Q)$ for every $i=1,\dots,6$.
    
\indent $\bullet$ If $\texttt{cond}_B(K)$ holds, then in the case where $[K:\Q]$ is odd, Remark~\ref{remark6} ensures that $\texttt{cond}_A(K)$ holds, so the argument reduces to the previous case. If $[K:\Q]<8$, the condition $\operatorname{rank}_{\mathbb{Z}} E_1^{\pm 1}(K) = 0$ implies that the only points in $C_i(K)$, for $i=1,\dots,6$, which do not already lie in $C_i(\Q)$ correspond to those for which the torsion subgroup of the associated elliptic curve grows over~$K$. Recall that, for an elliptic curve $E$ defined over $\Q$, there exist only finitely many number fields of a given degree such that $E(\Q)_{\mathrm{tors}} \neq E(K)_{\mathrm{tors}}$. In our setting, all such fields with $[K:\Q]<8$ are explicitly listed in the Appendix. In particular, under the assumption $\operatorname{rank}_{\mathbb{Z}} E_1^{\pm 1}(K) = 0$, the only fields of degree less than $8$ over~$\Q$ where the torsion subgroup of the curves under consideration may grow are the quadratic fields $\Q(\sqrt{-3})$, $\Q(\sqrt{-1})$, $\Q(\sqrt{2})$, and $\Q(\sqrt{3})$; and the quartic fields $\Q(\zeta_8)$, $ \Q(\zeta_{12})$, $\Q(\alpha)$, and $\Q(\beta)$, where the minimal polynomial of $\alpha$ is $x^4-3x^2+3$, and that of $\beta$ is $x^4-2x^3-2x+1$.
\end{remark}  
\vspace{2mm}
\fbox{$r\in K$}  Assume that \( r \notin K \). Since \( [L:K] = 2 \), it follows that \( L = K(r) \), and in particular, the minimal polynomial of \( r \) over \( K \) is a quadratic polynomial \( Q(x) \in K[x] \).  Thus, for \( i=1,2 \), we have  $\varphi_i(t,X_1,X_3)=(r, \alpha_i+\beta_i r)\in C_i(L)$, for certain \( \alpha_1,\alpha_2 ,\beta_1,\beta_2 \in K \). Hence, there exist \( r_{1}, r_{2} \in K \) such that  
\[
\begin{array}{r}
(x+4)\left(x^{2}+4 x+16\right)-(\alpha_1+\beta_1 x)^{2} = Q(x)\left(x-r_{1}\right),\\
x\left(x^{2}+4 x+16\right)-(\alpha_2+\beta_2 x)^{2} = Q(x)\left(x-r_{2}\right).
\end{array}
\]
Thus, \( \left(r_{i}, \alpha_i+\beta_i r_{i}\right)\in C_i(K) \). For \( P_1=(x_1,y_1)\in C_1(K) \) and \( P_2=(x_2,y_2)\in C_2(K) \), let \( Q_{\beta_{1}}(x),Q_{\beta_{2}}(x)\in K[x] \) be such that  
\[
\begin{array}{r}
(x+4)\left(x^{2}+4 x+16\right)-(y_1-\beta_1 x_1+\beta_1 x)^{2} = Q_{\beta_1}(x)\left(x-{x_{1}}\right),\\
x\left(x^{2}+4 x+16\right)-(y_2-\beta_2 x_2+\beta_2 x)^{2} = Q_{\beta_2}(x)\left(x-{ x_{2}}\right).
\end{array}
\]
We solve the system of quadratic equations arising from the equality \( Q_{\beta_1}(x) = Q_{\beta_2}(x) \). In the case where a solution \( \beta_1, \beta_2 \in K \) exists, we obtain a polynomial \( Q(x) \). We discard the case where \( \deg Q(x) = 1 \), since \( [K(r):K] = 2 \). 

Next, we impose the condition \( r = s^2 \) in \( L \). If such an element \( s \) exists, we compute \( t \) from the equality \( t^2 - s t - 1 = 0 \). Finally, we verify that \( t \in L \), and that both \( G(t) \) and \( G(-t) \) are squares in \( L \).

Now, thanks to Remark~\ref{rem2}, in the case where either $\texttt{cond}_A(K)$ holds or we are under $\texttt{cond}_B(K)$ with $[K:\Q]$ odd, it suffices to carry out the preceding computations and verifications over $K = \Q$. For the remaining cases where $\texttt{cond}_B(K)$ holds and $[K:\Q] < 8$, we refer to Table~\ref{growth}, specifically the columns corresponding to the cases $E_1$ and $E_1^{-1}$. In these situations, the analysis reduces to a finite list of number fields, namely:
\[
\Q, \quad \Q(\sqrt{-3}), \quad \Q(\sqrt{-1}), \quad \Q(\sqrt{3}), \quad \Q(\zeta_8), \quad \Q(\zeta_{12}), \quad \Q(\alpha), \quad \text{and} \quad \Q(\beta).
\]
For each of these fields, and for every possible pair of points $P_1 \in C_1(K)$ and $P_2 \in C_2(K)$, one verifies that \( r \in K \), except in the case $K = \Q(\zeta_{12})$, where we obtain $r \in \left\{ \tfrac{1}{2}(-15 \pm \sqrt{-15}) \right\}$. However, these exceptional values correspond to $t \in \mathcal T$, which are explicitly excluded in the statement of Proposition~\ref{prop}. Therefore, in all relevant cases, we conclude that \( r \in K \).

\vskip 2mm
$\fbox{$s\in K$}$ 
Assume $r\in K$ and $s \notin  K$. Since \( [L:K] = 2 \), it follows that \( L = K(s) \), and in particular, $r=s^2\notin K^2$. Therefore, it follows from \eqref{eqGundert2} that there exist $\mu, \psi \in  K$ such that $s^{2}\pm 2 s+4=(\mu\pm\psi s)^{2}$. Since $s^2=r\in  K$, we obtain
\begin{equation}\label{sinK}
r^{2}+4 r+16=\left(s^{2}-2 s+4\right)\left(s^{2}+2 s+4\right)=\left(\mu^{2}-s^2 \psi^{2}\right)^{2}=\left(\mu^{2}-r \psi^{2}\right)^{2} \in  K^{2}.
\end{equation}
 Since $t^{2}-s t-1=0$, we have $t=(s\pm \alpha)/2$, where $\alpha\in K$ satisfies $\alpha^2=s^2+4=r+4$.  We split the proof in two cases:
\begin{itemize}
  \item Case $r+4 \in K^2$. Combining this with \eqref{sinK}, we obtain $(r+4)\left(r^{2}+4 r+16\right) \in K^{2}$. Hence, {$\varphi_1(t,X_1,X_3)=(r,\left(t^{-1}+t^{-3}\right)X_1X_3))\in C_1(K)$}. 
  \item Case $r+4 \notin  K^2$. In this case, $L=K(s)=K(\alpha)$, implying \( s\alpha \in K \). Thus \( r(r + 4) \in K^{2} \). Combining this with \eqref{sinK}, we establish {$\varphi_3(t,X_1,X_3)=(r,\left(1-t^{-4}\right)X_1X_3)\in C_3(K)$}. 
\end{itemize}

As in the previous case, thanks to Remark~\ref{rem2}, the required computations reduce to analyzing Table~\ref{growth}, specifically the curves \( C_1 \) and \( C_3 \), which correspond to the elliptic curve \( E_1 \). In this setting, the field \( K \) is one of the following number fields:
\[
\Q, \quad \Q(\sqrt{-3}), \quad \Q(\sqrt{-1}), \quad \Q(\sqrt{3}), \quad \text{and} \quad \Q(\zeta_{12}).
\]
For each of these cases, and for every point in either \( C_1(K) \) or \( C_3(K) \), let \( r \) denote its first coordinate. From the identity \( r = \left( t - \frac{1}{t} \right)^2 \notin K^2 \), we solve for \( t \), and thus obtain \( L = K(s) = K(t) \). Finally, we verify that the conditions \( G(t) \in L^2 \) and \( G(-t) \in L^2 \) are never simultaneously satisfied. Hence, we conclude that \( s \in K \).

\vskip 2mm
$\fbox{$t\in K$}$  
Assume that $r,s\in K$ and $t \notin K$. Since $[L:K]=2$, it follows that $L=K(t)$. We have the following identities:
\[
\begin{cases}
t^{2}=s t+1,\\
t^{4}=(s^{2}+2) t^{2}-1,\\
t^{8}=(s^{4}+4s^{2}+2) t^{4}-1.
\end{cases}
\]
From the first equality we deduce that $t^{2} \notin K$, since $s \ne 0$, otherwise $t=\pm 1$. From the second one, we have $t^{4} \notin K$, because otherwise $s^2+2=0$ and $t$ would be a primitive 8th root of unity, say $t=\zeta_8$. In this case, $L=\mathbb{Q}(\zeta_8)$ but $G(\zeta_8) \notin L^2$. 

Therefore, if we denote $w=t^4$, we have $L=K(w)$, and by the third identity, the minimal polynomial of $w$ over $K$ is
\begin{equation}\label{eqF}
F(x)=x^{2}-(s^{4}+4s^{2}+2)x+1 \in K[x].
\end{equation}
On the other hand, we have $\varphi_4(t,X_1,X_3)=(t^4, t^2 X_1, X_3) \in C_4(L)$. Hence, there exist $\alpha, \beta \in K$ such that
\[
w(w^{2}+14w+1)=(\alpha+\beta w)^{2}.
\]
In other words, there exists $\omega_0 \in K$ such that
\[
x(x^{2}+14x+1)-(\alpha+\beta x)^{2}=F(x)(x-\omega_0).
\]
Therefore, $(\omega_0, \alpha+\beta \omega_0) \in C_4(K)$. 

For any point $P=(x_0,y_0) \in C_4(K)$, we have
\begin{equation}\label{eqFbeta}
x(x^{2}+14x+1)-(y_0-\beta x_0+\beta x)^{2}=F_{\beta}(x)(x-x_0).
\end{equation}

Consider the point $(0,0) \in C_4(K)$. Then we obtain $F_{\beta}(x)=x^2+(14-\beta^2)x+1$. Comparing with \eqref{eqF}, we deduce that $14-\beta^{2}=-(s^{4}+4s^{2}+2)$, so that $
s^{4}+4s^{2}+16=\beta^{2}$. Thus, $(s,\beta) \in C_5(K)$.

Let $P=(x_0,y_0) \in C_4(K)$, with $P \ne (0,0)$. Calculating $F_{\beta}(x)$ in \eqref{eqFbeta} and comparing with \eqref{eqF}, we obtain the following system of equations:
\[
\begin{cases}
\beta^2 x_0-2\beta y_0+x_0^2+14x_0=0,\\
s^4+4s^2+x_0+16-\beta^2=0.
\end{cases}
\]
In particular, it must have solutions $\beta, s \in K$. In this case, the polynomial $x^2-s x-1$ must be irreducible over $K$. Finally, we need to verify that both $G(t)$ and $G(-t)$ are squares in $L=K(t)$.

Once again, Remark~\ref{rem2} ensures that, when considering the case \( (0,0) \in C_4 \), the necessary computations and verifications must be performed using the points on the curve \( C_5(K) \), over a finite number of number fields \( K \). According to Table~\ref{growth}, in the column corresponding to this curve, namely, \( E_6^{-1} \), the relevant number fields are
\[
\Q, \quad \Q(\sqrt{-3}), \quad \Q(\sqrt{-1}), \quad \Q(\sqrt{3}), \quad \Q(\zeta_8), \quad \Q(\zeta_{12}), \quad \text{and} \quad \Q(\beta).
\]
In all of these cases, we verify that \( t \in K \).

Finally, when \( P \neq (0,0) \), we refer to Table~\ref{growth}, specifically the column corresponding to the curve \( C_4 \), that is, \( E_4 \). The relevant number fields are
\[
\Q, \quad \Q(\sqrt{2}), \quad \text{and} \quad \Q(\sqrt{3}).
\]
Again, we conclude that in all these cases, \( t \in K \).

This completes the proof that $t\in K$.

\vskip 2mm
\fbox{Final step} Assume that $r,s,t\in K$. Since $[L:K]=2$, there exists $\alpha \notin K$ such that $\alpha^2 \in K$ and $L=K(\alpha)$. By \eqref{eq2C}, there exist $\gamma_1, \delta_1, \gamma_3, \delta_3 \in K$ such that $X_1 = \delta_{1} + \gamma_{1}\alpha$ and $X_3 = \delta_{3} + \gamma_{3}\alpha$, and we have
$$
\begin{cases}
 G(t) = \left(\delta_{1} + \gamma_{1}\alpha\right)^{2}, \\
 G(-t) = \left(\delta_{3} + \gamma_{3}\alpha\right)^{2}.
\end{cases}
$$
Since $t \in K$ and $\alpha \notin K$, we deduce that $\gamma_{1}\delta_{1} = \gamma_{3}\delta_{3} = 0$. Thus, there are four possible cases:

\begin{itemize}
  \item[$\bullet$] $\gamma_{1} = \gamma_{3} = 0$. In this case, $\varphi_6(t,X_1,X_3) = \left(t^{2}, \delta_{1} \delta_{3}\right) \in C_6(K)$. 

  \item[$\bullet$] $\delta_{1} = \delta_{3} = 0$. Equivalently, $\varphi_6(t,X_1,X_3) = \left(t^{2}, \alpha^2 \gamma_{1} \gamma_{3}\right) \in C_6(K)$.

  \item[$\bullet$] $\gamma_{1} = \delta_{3} = 0$. We have $\varphi_0(t,X_1,X_3) = \left(t, \delta_{1}\right) \in C_0(K)$. In particular, $X_1=\delta_1$.

  \item[$\bullet$] $\delta_{1} = \gamma_{3} = 0$. We obtain $\varphi'_0(t,X_1,X_3) = \left(-t, \delta_{3}\right) \in C_0(K)$. In particular, $X_3=\delta_3$.
\end{itemize}

To conclude the proof of Proposition~\ref{prop}, it remains to show that the cases \( \gamma_{1} = \gamma_{3} = 0 \) and \( \delta_{1} = \delta_{3} = 0 \) are not possible. Once again, invoking Remark~\ref{rem2}, this reduces to analyzing the column corresponding to \( C_6 \) in Table~\ref{growth}, that is, the elliptic curve \( E_6 \). The relevant number fields in which the computations must be carried out are $\Q$, $\Q(\zeta_8)$, and $\Q(\zeta_{12})$. For each of these fields, and for every point \( (x_0, y_0) \in C_6(K) \), we must verify whether \( x_0 = t^2 \in K^2 \). Moreover, among all such cases, we discard those yielding trivial points on \( \mathcal{C} \), for instance, those with \( x_0 \in \{0, \pm 1\} \) or \( y_0 = 0 \). No admissible cases arise from this analysis. This completes the proof of Proposition~\ref{prop}.
 
\section{Proof of Theorem \ref{quad_points}}\label{sec_proof_quad_points}

Let $K$ be a number field, and suppose that either $\texttt{cond}_A(K)$ or $\texttt{cond}_B(K)$ holds. Let $L$ be a quadratic extension of $K$, and let $(t, X_1, X_3) \in \mathcal{C}(L)$ be a non-trivial point. We divide the proof according to whether $t \in \mathcal{T}$ or not.

First, assume that $t \notin \mathcal{T}$. Then Proposition~\ref{prop} allows us to conclude that either $(t, X_1) \in C_0(K)$ or $(-t, X_3) \in C_0(K)$. Assume further that the class number of $K$ is $1$. In the case where $(t, X_1) \in C_0(K)$, there exist elements $u = X_1$, $v$, $m \in K$ such that $G(t) = u^2$ and $G(-t) = m v^2$. On the other hand, if $(-t, X_3) \in C_0(K)$, there exist elements $v = X_3$, $u$, $n \in K$ such that $G(t) = n u^2$ and $G(-t) = v^2$. 

Thanks to the parametrization~\eqref{eqPara} given in Section~\ref{sec_setting}, we deduce that $[X_0 : X_1 : X_2 : X_3 : X_4]$ belongs to $\mathcal{P}_1(t, u)$ if $(t, X_1) \in C_0(K)$, and belongs to $\mathcal{P}_3(-t, v)$ if $(-t, X_3) \in C_0(K)$. This proves part~(i) of Theorem~\ref{quad_points}.

Moreover, if $E_0(K) = E_0(\mathbb{Q})$, then $C_0(K) = C_0(\mathbb{Q})$, and therefore $t, X_1 \in \mathbb{Q}$ or $t, X_3 \in \mathbb{Q}$, which proves part~(ii). This completes the proof of Theorem~\ref{quad_points} in the case $t \notin \mathcal{T}$.

We now assume that $t \in \mathcal{T}$. Then the corresponding points is the set $\mathcal{P}_{\mathcal{T}}$ (see Table~\ref{otherpoints}). In particular, we have $\sqrt{-1}, \sqrt{3}, \sqrt{5} \in L$. Therefore, $K$ must be contained in $\mathbb{Q}(\sqrt{-1}, \sqrt{3}, \sqrt{5})$, and $[L : K] \leq 2$. Among all possible number fields $K$ satisfying these properties, the only one for which either $\texttt{cond}_A(K)$ or $\texttt{cond}_B(K)$ holds is $K = \mathbb{Q}(\zeta_{12})$. Thus, we have $\mathcal{P}_{\mathcal{T}} \subset \mathcal{C}(K(\sqrt{5})) \subset \Gamma_2^*(\mathcal{C}, \mathbb{Q}(\zeta_{12}))$. Moreover, since $E_0(\mathbb{Q}(\zeta_{12})) = E_0(\mathbb{Q})$, a similar argument to the proof of (ii) allows us to conclude that
\[
\Gamma_2^*(\mathcal{C}, \mathbb{Q}(\zeta_{12})) = \mathcal{S}_1^*(\mathbb{Q}) \cup \mathcal{S}_3^*(\mathbb{Q}) \cup \mathcal{P}_{\mathcal{T}}.
\]
This finishes the proof of Theorem~\ref{quad_points}.

\section{Proof of Theorem \ref{main}}\label{sec_proof_main}

Let \( K \) be a number field with \( K \ne \Q(\zeta_{12}) \), and suppose that either \( \texttt{cond}_A(K) \) or \( \texttt{cond}_B(K) \) holds. Let \( L \) be a quadratic extension of \( K \). As seen in Section~\ref{sec_setting}, any arithmetic progression of five squares $(a_1^2,a_2^2,a_3^2,a_4^2,a_5^2)$ with $a_1,a_2,a_3,a_4,a_5 \in L$ corresponds to a point on $\mathcal{C}(L)$, and vice versa.  

Now assume that the class number of $K$ is $1$. Using the characterization provided in Theorem~\ref{quad_points}, we can analyze the possible arithmetic progressions of five squares. That is, for each point $(t, u) \in C_0(K)$, we obtain a progression arising from $\mathcal{P}_1(t,u)$ of the form:
\[
((t^{2}-2t-1)^2,\, u^2,\, (t^{2}+1)^2,\, mv^2,\, (t^{2}+2t-1)^2).
\]
{ Note that if \( m \in K^2 \), then the above arithmetic progression is already defined over \( K \); hence, it is not properly defined over any quadratic extension of \( K \) }. For each point $(-t, v) \in C_0(K)$, we obtain a progression arising from $\mathcal{P}_3(-t, v)$ of the form:
\[
((t^{2}+2t-1)^2,\, nu^2,\, (t^{2}+1)^2,\, v^2,\, (t^{2}-2t-1)^2).
\]
{ (The case \( n \in K^2 \) is analogous to the case $m\in K^2$, and the same conclusion follows).}
Note this last arithmetic progression is equivalent to its reverse, and hence any non-elementary arithmetic progression of five squares properly defined over $L$ is, up to equivalence, of the form $\left(a_1^2,a_2^2,a_3^2,\alpha a_4^2,a_5^2\right)$, where $a_1,a_2,a_3,a_4,a_5, \alpha \in K$ and $\alpha$ is a non-square in $K$. This completes the proof of part (i).

We now prove part (ii). Suppose that $K \ne \Q$ and \( E_0(K) = E_0(\Q) \). As in the previous case, all relevant points satisfy $(t, u) \in C_0(\Q)$ or $(-t, v) \in C_0(\Q)$. Therefore, any corresponding non-elementary arithmetic progression of five squares is defined over $\Q(\sqrt{m})$ for some square-free $m \in \Q$, which is not properly defined over $L$. This proves part (ii).

\section{Proof of Theorem \ref{Qz12}}\label{sec_proof_Qz12}

In Theorem~\ref{quad_points}, we proved that $\Gamma_2^*(\mathcal{C}, \mathbb{Q}(\zeta_{12})) = \mathcal{S}_1^*(\mathbb{Q}) \cup \mathcal{S}_3^*(\mathbb{Q}) \cup \mathcal{P}_{\mathcal{T}}$. Hence, by an argument analogous to that used in the proof of Theorem~\ref{main}~(ii), any non-elementary arithmetic progression of five squares arising from points in \( \mathcal{S}_1^*(\mathbb{Q}) \cup \mathcal{S}_3^*(\mathbb{Q}) \) is defined over a quadratic extension of \( \mathbb{Q} \), and therefore is not properly defined over a quadratic extension of $\mathbb{Q}(\zeta_{12})$.

It remains to consider the points in \( \mathcal{P}_{\mathcal{T}} \). From these, we obtain the following arithmetic progressions of five squares defined over $L= \mathbb{Q}(\zeta_{12}, \sqrt{5})$:
\[
(3, 1, -1, -3, -5), \quad (-5, -3, -1, 1, 3), \quad (-3, -1, 1, 3, 5), \quad (5, 3,1, -1, -3).
\]
All of these are equivalent, over \( L  \), to the progression $(-3, -1, 1, 3, 5)$. This completes the proof.

\section{Proof of Theorem \ref{six}}\label{sec_proof_six}

Let \( K \) be a number field satisfying either \( \texttt{cond}_A(K) \) or \( \texttt{cond}_B(K) \), and let \( L \) be a quadratic extension of \( K \). Theorem \ref{six} follows as a consequence of Theorems \ref{main} and \ref{Qz12}.

We begin by proving that if \( K \ne \Q(\zeta_{12}) \), then no non-elementary arithmetic progression of five squares properly defined over \( L \) can be extended to a progression of six squares.

Suppose, that there exists a non-elementary arithmetic progression of six squares properly defined over \( L \), say $(a_1^2,a_2^2,a_3^2,a_4^2,a_5^2, a_6^2)$. {Then $(a_1^2,a_2^2,a_3^2,a_4^2,a_5^2)$ or $(a_2^2,a_3^2,a_4^2,a_5^2,a_6^2)$ is a }non-elementary arithmetic progressions of five squares properly defined over \( L \). We now consider the different possibilities for \( K \):
\begin{itemize}
\item If \( K = \Q \), then it is well known that no non-constant arithmetic progression of six squares exists over any quadratic extension of \( \Q \) (cf. \cite{X}).

\item If \( K \ne \Q \) and \( E_0(K) = E_0(\Q) \), then Theorem \ref{main}~(ii) implies that no non-elementary arithmetic progression of five squares can be properly defined over \( L \), hence such a progression of length six cannot exist.

\item Otherwise, we assume that the class number of \( K \) is one. In this setting, Theorem \ref{main}~(i) applies, and any non-elementary arithmetic progression of five squares properly defined over \( L \) is equivalent to a progression of the form $\left(x_1^2, x_2^2, x_3^2, \alpha \, x_4^2, x_5^2\right)$, where \( x_1, x_2, x_3, x_4, x_5, \alpha \in K \), and \( \alpha \in K \) is a non-square.

Now, assume that one of the two five-tuples, say \( (a_1^2,a_2^2,a_3^2,a_4^2,a_5^2) \), is equivalent to such a progression. That is, there exists \( s \in K^\times \) such that $(s a_1^2,s a_2^2,s a_3^2,s a_4^2,s a_5^2) = (x_1^2, x_2^2, x_3^2, \alpha \, x_4^2,$ $ x_5^2)$. From the identities \( sa_1^2 = x_1^2 \) and \( s a_4^2 = \alpha \, x_4^2 \), we deduce that \( \alpha \in K^2 \), contradicting the assumption that \( \alpha \) is not a square in \( K \). Thus, no such progression of length six exists.
\end{itemize}

It remains to consider the case \( K = \Q(\zeta_{12}) \). As shown in Theorem \ref{Qz12}, the only non-elementary arithmetic progression of five squares properly defined over a quadratic extension of \( \Q(\zeta_{12}) \), up to equivalence, is $(-3, -1, 1, 3, 5)$, 
which is defined over \( L = \Q(\zeta_{12}, \sqrt{5}) \).

Any non-elementary arithmetic progression of six squares properly defined over \( L \) must extend this progression. This yields two natural candidates: $(-5, -3, -1, 1, 3, 5)$ and $(-3, -1, 1, 3, 5, 7)$. However, only the first of these is defined over \( \Q(\zeta_{12}, \sqrt{5}) \). Therefore, up to equivalence, this is the unique non-elementary arithmetic progression of six squares properly defined over a quadratic extension of \( \Q(\zeta_{12}) \).

\section*{Appendix}
This appendix collects all the necessary information required to carry out the computations performed in this article. In particular, we include the basic data concerning the curves and number fields that appear throughout this work.

For each curve \(C_i\), \(i=0,1,\dots,6\), the second row of Table~\ref{TableEC} displays the corresponding \(\Q\)-isomorphic elliptic curve, while the third row shows its label in the LMFDB database~\cite{lmfdb}.

\begin{longtblr}
[caption={Elliptic curves}, label={TableEC}]
{hlines,vlines,colspec  =cccccc }
$C_0$   & $C_1$   & $C_2$ &  $C_3$  & $C_4$ & $C_5$ & $C_6$ \\
$E_0$   & $E_1$   & $E^{-1}_1$ & $E_1$ & $E_4$ &  $E^{-1}_6$  & $E_6$\\
\href{https://www.lmfdb.org/EllipticCurve/Q/192a2/}{\texttt{192.a2}} &  
\href{https://www.lmfdb.org/EllipticCurve/Q/24a4/}{\texttt{24.a5}} &  
\href{https://www.lmfdb.org/EllipticCurve/Q/48a4/}{\texttt{48.a5}} &   
\href{https://www.lmfdb.org/EllipticCurve/Q/24a4/}{\texttt{24.a5}} &  
\href{https://www.lmfdb.org/EllipticCurve/Q/24a3/}{\texttt{24.a2}} &  
\href{https://www.lmfdb.org/EllipticCurve/Q/48a1/}{\texttt{48.a4}}  & 
\href{https://www.lmfdb.org/EllipticCurve/Q/24a1/}{\texttt{24.a4}} \\
\end{longtblr}

Note that the elliptic curves \(E_1, E_4\), and \(E_6\) belong to the same \(\Q\)-isogeny class. Therefore, for any number field \(K\), we have \(\operatorname{rank}_{\Z} E_1(K) = \operatorname{rank}_{\Z} E_4(K) = \operatorname{rank}_{\Z} E_6(K)\). In particular, the ranks of any of their quadratic twists also coincide.

Assuming that \(\operatorname{rank}_{\Z} E_1^{\pm 1}(K) = 0\) for a given number field \(K\), computing \(C_i(K)\) for \(k=i, \dots, 6\) reduces to computing the torsion subgroup over \(K\). 

{ In our particular setting, given an elliptic curve \(E\) defined over \(\Q\), there are only finitely many number fields of fixed degree \(d\) such that \(E(\Q)_{\mathrm{tors}} \neq E(K)_{\mathrm{tors}}\). If the conductor of \(E\) is less than \(400.000\), these number fields appear\footnote{Torsion growth data was computed by the author and Filip Najman~\cite{GJN2}.} in the LMFDB database~\cite{lmfdb} for $d<24$.}

Table~\ref{growth} illustrates the torsion growth for the elliptic curves considered in this article. The first column lists all number fields of degree \(<8\) where the torsion grows. For each such field, the subsequent columns display the corresponding torsion subgroup over that field. We use the notation \([n]\) to denote the cyclic group of order \(n\), and \([n,m]\) to denote the product of cyclic groups of orders \(n\) and \(m\). Let \(\alpha \in \overline{\Q}\) be a root of the irreducible polynomial \(x^4 - 3x^2 + 3\), and let \(\beta \in \overline{\Q}\) be a root of the irreducible polynomial \(x^4 - 2x^3 - 2x + 1\).

{
For example, the elliptic curve $E_0$ has LMFDB label \texttt{192.a2}. Therefore, to obtain the relevant information about the growth of the torsion subgroup of this elliptic curve, it suffices to consult the corresponding LMFDB webpage:
\begin{center}
\href{https://www.lmfdb.org/EllipticCurve/Q/192.a2/}{https://www.lmfdb.org/EllipticCurve/Q/192.a2/}
\end{center}
and to examine the section entitled \emph{``Growth of torsion in number fields''}, which lists the number fields over which the torsion of this elliptic curve grows. In particular, this webpage shows that if the torsion of $E_0$ grows over a number field $K$ of degree less than~8, then $K$ must be one of
\[
\mathbb{Q}(\sqrt{-2}), \quad \mathbb{Q}(\sqrt{2},\sqrt{3}), \quad \text{or} \quad \mathbb{Q}(\sqrt{-1}, \sqrt{6}),
\]
and in all of these fields the torsion subgroup satisfies $E_0(K)_{\mathrm{tors}} \simeq \mathbb{Z}/2\mathbb{Z} \times \mathbb{Z}/4\mathbb{Z}$.
}

\begin{longtblr}
[label={growth}, caption={Primitive Torsion Growth over number field of degree $< 8$}]
{cells = {mode=imath},hlines,vlines,
colspec=ccccccc,rowhead=1}
   K   & E_1(K)_{\mathrm{tors}} & E^{-1}_1(K)_{\mathrm{tors}} & E_4(K)_{\mathrm{tors}} & E_6(K)_{\mathrm{tors}} & E^{-1}_6(K)_{\mathrm{tors}} & E_0(K)_{\mathrm{tors}} \\
\hline
\Q  & [4]  & [2] & [4] & [2,4] & [2,2] &[2,2] \\
\hline 
\Q(\sqrt{-3}) & [2,4] & [2,2] & - & - & [2,4] & -\\
\Q(\sqrt{-2}) & - & - & - & - & - & [2,4] \\
\Q(\sqrt{-1}) & [8] & [8] & - & - & [2,4] & - \\
\Q(\sqrt{2}) & - & - & [8] & - & - & - \\
\Q(\sqrt{3}) & [8] & [4] & [2,4] & - &  [2,4] & - \\
\Q(\sqrt{6}) & - & - & [8] & - & - & - \\
\hline
\Q(\zeta_8) & - & - & - & [2,8] & [2,8] & - \\
\Q(\zeta_{12}) & [2,8] & [2,8] & - & [4,4] & [4,4] & - \\
\Q(\alpha) & - & [2,4] & - & - & - & - \\
\Q(\beta) & - & [8] & - & - & [2,8] & - \\
\Q(\sqrt{2},\sqrt{3}) & - & - & [2,8] & [2,8] & - & [2,4]\\
\Q(\sqrt{-1},\sqrt{6}) & - & - & - & - & - & [2,4] \\
\end{longtblr}

Note that for those elliptic curves and number fields $K$, we have $\operatorname{rank}_{\Z}E_1(K)\ne 0$ only if $\sqrt{6}\in K$;  $\operatorname{rank}_{\Z}E_1^{-1}(K)\ne 0$ only if $K=\Q(i,\sqrt{6})$; and $\operatorname{rank}_{\Z}E_0(K)\ne 1$ only if $K=\Q(i,\sqrt{6})$;. Finally, for those number fields $K$ of degree $4$, such that $\operatorname{rank}_{\Z}E^{\pm 1}_1(K)= 0$, Tables \ref{pointscyclo} and \ref{pointsNOcyclo}  show  $C_i(K)$ for $i=1,\dots,6$.

The following diagram illustrates the lattice of number fields associated with the growth of the torsion of elliptic curves but does not include those in which the rank increases for any of the elliptic curves.

\begin{center}
\begin{tikzpicture}[
    node distance=1.3cm and 1.3cm, 
    every node/.style={inner sep=1pt}, 
    every path/.style={draw} 
    ]
    \node (Q) at (0,0) {$\mathbb{Q}$};  
    \node (Qi)  at (0,1) {$\mathbb{Q}(\sqrt{-1})$};
           \node (Qz8) at (-1,2) {$\mathbb{Q}(\zeta_8)$}; 
   \node (Qm2) at (-2,1) {$\mathbb{Q}(\sqrt{-2})$};
             \node (Qalpha) at (-3,2) {$\mathbb{Q}(\alpha)$}; 
    \node (Qm3) at (-4,1) {$\mathbb{Q}(\sqrt{-3})$};
            \node (Qz12) at (1,2) {$\mathbb{Q}(\zeta_{12})$}; 
   \node (Q2)  at (2,1) {$\mathbb{Q}(\sqrt{2})$};
    \node (Q3)  at (4,1) {$\mathbb{Q}(\sqrt{3})$};
          \node (Qbeta) at (3,2) {$\mathbb{Q}(\beta)$}; 
    \draw (Q) -- (Q2);
    \draw (Q) -- (Q3);
    \draw (Q) -- (Qm2);
    \draw (Q) -- (Qm3);
    \draw (Q) -- (Qi);
    
    \draw (Qi) -- (Qz8);
    \draw (Qm2) -- (Qz8);
    \draw (Q2) -- (Qz8);

       \draw (Qi) -- (Qz12);
     \draw (Qm3) -- (Qz12);
     \draw (Q3) -- (Qz12);
    
     \draw (Q3) -- (Qbeta);
     \draw (Qm3) -- (Qalpha);
\end{tikzpicture}
\end{center}

\begin{longtblr}
[caption = {Points on the curves over $\Q(\zeta_8)$ and $\Q(\zeta_{12})$}, label=pointscyclo]
{cells = {mode=imath},hlines,vlines,colspec  = ccccc,rowhead=1}
 K & \Q(\zeta_8) & \Q(\zeta_{12})  \\ 
 
C_1(K)&
 \begin{array}{c}
    (-4, 0),(0, -8),\\ 
 (\pm 4 \zeta_{8}^2 - 4, 8)
   \end{array}&
 \begin{array}{c} 
      (-4, 0),(0, -8),\\ 
 (-4 \zeta_{12}^2, 0),\,(4 \zeta_{12}^2 - 4, 0),\, (-8, 16 \zeta_{12}^2 - 8),\\
  (\pm 4 \zeta_{12}^3 - 4, 8),\\ 
  (-4 \zeta_{12}^3 + 8 \zeta_{12} + 4, 16 \zeta_{12}^3 - 32 \zeta_{12} - 24),\\ 
    (4 \zeta_{12}^3 - 8 \zeta_{12} + 4, 16 \zeta_{12}^3 - 32 \zeta_{12} + 24)
   \end{array} \\  
C_2(K)&
 \begin{array}{c}
 (0, 0),\\
  (\pm 4 \zeta_{8}^2, 8 \zeta_{8}^2),\, (-4, -8 \zeta_{8}^2)
  \end{array}&
 \begin{array}{c} 
  (0, 0),\\
   (4 \zeta_{12}^2 - 4, 0),\, (-4 \zeta_{12}^2, 0),\\ 
   (-4 \zeta_{12}^3, -8 \zeta_{12}^3),\,   (-4, -8 \zeta_{12}^3),\, (4 \zeta_{12}^3, 8 \zeta_{12}^3),\\ 
   (4, -8 \zeta_{12}^3 + 16 \zeta_{12}),\\ 
    (4 \zeta_{12}^3 - 8 \zeta_{12} - 8, -24 \zeta_{12}^3 - 32 \zeta_{12}^2 + 16),\\ 
   (-4 \zeta_{12}^3 + 8 \zeta_{12} - 8, 24 \zeta_{12}^3 - 32 \zeta_{12}^2 + 16)
  \end{array}
 \\  
C_3(K)&
 \begin{array}{c}
 (0, 0),\, (-4, 0),\\ 
 (\pm 2 \zeta_{8}^2 - 2, 8 \zeta_{8}^2)\\
  \end{array}&
 \begin{array}{c} 
 (0, 0),\, (-4, 0),\\ 
 (-4 \zeta_{12}^2, 0),\, (4 \zeta_{12}^2 - 4, 0),\, (-2, 8 \zeta_{12}^2 - 4),\\ 
 (\pm 2 \zeta_{12}^3 - 2, 8 \zeta_{12}^3),\\ 
 (-2 \zeta_{12}^3 + 4 \zeta_{12} - 2, 8 \zeta_{12}^3 - 16 \zeta_{12}),\\ 
 (2 \zeta_{12}^3 - 4 \zeta_{12} - 2, 8 \zeta_{12}^3 - 16 \zeta_{12})
 \end{array}
  \\  
C_4(K)&
 \begin{array}{c}
  (0, 0)\,, (1, 4),\\ 
 (-2 \zeta_{8}^3 + 2 \zeta_{8} - 3, -6 \zeta_{8}^3 + 6 \zeta_{8} - 8),\\ 
 (2 \zeta_{8}^3 - 2 \zeta_{8} - 3, -6 \zeta_{8}^3 + 6 \zeta_{8} + 8),\\ 
 \end{array}&
 \begin{array}{c} 
 (0, 0)\,, (1, 4),\\
  (-1, -2 \zeta_{12}^3 + 4 \zeta_{12}),\\ 
   (\pm(4 \zeta_{12}^3 - 8 \zeta_{12} - 7), 0)
 \end{array}
  \\  
C_5(K)&
 \begin{array}{c}
 (0, 4),\\ 
  (\pm 2 \zeta_{8}^2, 4),\\
    (\pm 2 \zeta_{8}^3, 4 \zeta_{8}^3),\, (\pm 2 \zeta_{8}, 4 \zeta_{8}),\\
    \end{array}&
 \begin{array}{c}
   (0, 4),\\ 
     (\pm 2 \zeta_{12}^2, 0),\,  (\pm(2 \zeta_{12}^2 - 2), 0),\\ 
    (\pm 2 \zeta_{12}^3, 4),\\
     (\pm 2, -4 \zeta_{12}^3 + 8 \zeta_{12})
    \end{array}
  \\ 
  
C_6(K)&
 \begin{array}{c}
 (0, 1),\, (1, 4),\, (-1, 4),\\ 
  (\pm(\zeta_{8}^3 - \zeta_{8}^2 + \zeta_{8}), 2 \zeta_{8}^3 - 4 \zeta_{8}^2 + 2 \zeta_{8}),\\ 
  (\pm(\zeta_{8}^3 + \zeta_{8}^2 + \zeta_{8}), -2 \zeta_{8}^3 - 4 \zeta_{8}^2 - 2 \zeta_{8})
    \end{array}
 & 
 \begin{array}{c} 
 (0, 1),\, (1, 4),\, (-1, 4)\\ 
   (\pm(2 \zeta_{12}^3 + 2 \zeta_{12}^2 - 1), 0),\\
  (\pm(2 \zeta_{12}^3 - 2 \zeta_{12}^2 + 1), 0),\\ 
   (\pm \zeta_{12}^3, 4 \zeta_{12}^2 - 2)
    \end{array}
\end{longtblr}

\begin{longtblr}
[caption = {Points on the curves over $\Q(\alpha)$ and $\Q(\beta)$}, label=pointsNOcyclo]
{cells = {mode=imath},hlines,vlines,colspec  = ccccc,rowhead=1}
 K & \Q(\alpha) & \Q(\beta)  \\ 
 
C_1(K) 
& \begin{array}{c}
(-4, 0)\,,(0, -8),\\
(-4 \alpha^2 + 4, 0),\\ 
(-8, 16 \alpha^2 - 24),\\ 
(4 \alpha^2 - 8, 0)
\end{array}
 & \begin{array}{c}
 (-4, 0)\,,(0, -8),\\
\!\!\!\!\!\!\!\!\!\! \!\!\!\!\!\! \!\!\!\!\!\!\!\! \!\!\!\!\!\!\!\! \!\!\!\!\!\!\!\! \!\!\!\!\!\!\!\! \!\!\!\!\!\!\!\! 
(-4 \beta^3 + 8 \beta^2 + 4 \beta + 8,\\ \qquad \qquad \qquad \quad \quad   -16 \beta^3 + 32 \beta^2 + 16 \beta + 40),\\
 \!\!\!\!\!\!\!\!\!\! \!\!\!\!\!\! \!\!\!\!\!\!\!\! \!\!\!\!\!\!\!\! \!\!\!\!\!\!\!\! \!\!\!\!\!\!\!\! \!\!\!\!\!\!\!\! 
 (4 \beta^3 - 8 \beta^2 - 4 \beta, \\ \qquad \qquad \qquad -16 \beta^3 + 32 \beta^2 + 16 \beta - 8)\end{array} 
 \\  
C_2(K)
 & \begin{array}{c}
 (0, 0),\\
   (-4 \alpha^2 + 4, 0),\,(4 \alpha^2 - 8, 0),\\
  (-4 \alpha^3 - 4 \alpha^2 + 4 \alpha + 4, -8 \alpha^2 - 8 \alpha),\\ 
  (4 \alpha^3 - 4 \alpha^2 - 4 \alpha + 4, -8 \alpha^2 + 8 \alpha)
  \end{array}
 & \begin{array}{c}
 (0, 0),\\
 (4 \beta^2, -8 \beta^3 - 8 \beta + 8),\\
  (4, 8 \beta^3 - 16 \beta^2 - 8 \beta - 8),\\ 
  \!\!\!\!\!\!\!\!\!\! \!\!\!\!\!\! \!\!\!\!\!\!\!\! \!\!\!\!\!\!\!\! \!\!\!\!\!\!\!\! \!\!\!\!\!\!\!\! \!\!\!\!\!\!\!\! 
(-8 \beta^3 + 12 \beta^2 + 8 \beta + 16,  \\ \qquad \qquad \qquad -40 \beta^3 + 64 \beta^2 + 24 \beta + 88)
  \end{array} 
 \\  
C_3(K)
& \begin{array}{c}
(0, 0),\, (-4, 0),\\
 (-4 \alpha^2 + 4, 0)\,,(4 \alpha^2 - 8, 0), \\(-2, -8 \alpha^2 + 12)
 \end{array}
 & \begin{array}{c}
 (0, 0)\,, (-4, 0),\\ 
 (2 \beta^3 - 4 \beta^2 - 2 \beta - 4, 8 \beta^3 - 16 \beta^2 - 8 \beta - 8),\\ 
 (-2 \beta^3 + 4 \beta^2 + 2 \beta, 8 \beta^3 - 16 \beta^2 - 8 \beta - 8)
 \end{array} 
 \\  
C_4(K)
& \begin{array}{c}
(1, -4)\,, (0, 0)
\end{array}
 & 
 \begin{array}{c}
 (1, -4)\,, (0, 0),\\
(-4 \beta^3 + 8 \beta^2 + 4 \beta - 3, 0),\\ 
 (-1, -2 \beta^3 + 4 \beta^2 + 2 \beta + 2),\\
  (4 \beta^3 - 8 \beta^2 - 4 \beta - 11, 0)\end{array}
  \\  
C_5(K) & 
\begin{array}{c}
(0, 4),\\
(\pm(2 \alpha^2 - 4), 0),\\ 
(\pm(2 \alpha^2 - 2), 0)
\end{array}
  & \begin{array}{c}
  (0, 4),\\ 
   (\pm 2, -4 \beta^3 + 8 \beta^2 + 4 \beta + 4),\\
         (\pm 2 \beta, -4 \beta^3 + 4 \beta^2 + 4),\\
    (\pm(2 \beta^3 - 4 \beta^2 - 4), 8 \beta^3 - 12 \beta^2 - 4 \beta - 20)
     \end{array} 
 \\  
C_6(K)
& \begin{array}{c}(0, 1)\,,(1, 4)\,, (-1, 4)\end{array}& \begin{array}{c}(0, 1)\,, (1, 4)\,, (-1, 4)\end{array} 
 \\  
\end{longtblr}

{
\begin{ack}
The author thanks the referee for helpful comments and suggestions that improved the paper.
\end{ack}
}

\end{document}